\documentclass[11pt]{amsart}

\usepackage{enumerate, url, parskip, bm}
\usepackage{amssymb}
\usepackage{amsmath,amscd}
\usepackage{amsthm}
\usepackage{graphicx}
 \usepackage{verbatim}
 \usepackage{color}
\usepackage{amsmath}
\usepackage[utf8]{inputenc}
\usepackage[T1]{fontenc}
\usepackage{hyperref}

\usepackage{
  etex,  
  lmodern,
  calc,
  amssymb,
  amsthm,
  tabularx,
  multicol,
  graphicx,  
}

\usepackage[all]{xy}

\def\co{\colon\thinspace}

\newcommand{\RR}{\mathbb{R}}
\newcommand{\QQ}{\mathbb{Q}}
\newcommand{\CC}{\mathbb{C}}

\newcommand{\sph}{\mathbb{S}}

\newcommand{\ZZ}{\mathbb{Z}}

\newcommand{\Diff}{\operatorname{Diff}}

\newtheorem{theorem}{Theorem}[section]

\newtheorem{problem}[theorem]{Problem}

\newtheorem{lemma}[theorem]{Lemma}

\newtheorem{prop}[theorem]{Proposition}
\newtheorem{remark}[theorem]{Remark}

\theoremstyle{definition}

\makeatletter
\def\equalsfill{$\m@th\mathord=\mkern-7mu
\cleaders\hbox{$\!\mathord=\!$}\hfill
\mkern-7mu\mathord=$}
\makeatother

\begin{document}

\abovedisplayskip=6pt plus3pt minus3pt
\belowdisplayskip=6pt plus3pt minus3pt

\title[Diffeomorphic souls and disconnected moduli spaces]{Diffeomorphic souls and disconnected moduli spaces of nonnegatively curved metrics}

\author{Igor Belegradek}
\address{Igor Belegradek\\ School of Mathematics\\ Georgia Institute of
Technology\\ Atlanta, GA, USA 30332}
\email{ib@math.gatech.edu}
\urladdr{www.math.gatech.edu/~ib}

\author[David Gonz\'alez-\'Alvaro]{David Gonz\'alez-\'Alvaro}
\address{David Gonz\'alez-\'Alvaro\\ ETSI de Caminos, Canales y Puertos\\ Universidad Polit\'ecnica de Madrid\\ 28040 Spain}
\email{david.gonzalez.alvaro@upm.es}
\urladdr{https://dcain.etsin.upm.es/~david/}

\thanks{2010 \it  Mathematics Subject classification.\rm\ 
Primary 53C20.}
\keywords{Nonnegative curvature, soul, moduli space, positive scalar curvature.}
\thanks{This work was partially supported by the Simons Foundation grant 524838 (Belegradek) and by MINECO grant MTM2017-85934-C3-2-P (Gonz\'alez-\'Alvaro).}

\begin{abstract} 
We give examples of open manifolds that carry infinitely many complete metrics of nonnegative 
sectional curvature such that they all have the same soul, and their isometry classes 
lie in different connected components of the moduli space. 
All previously known examples of this kind have souls of codimension one.
In our examples the souls have codimensions three and two. 
\end{abstract}

\maketitle

\thispagestyle{empty}

\section{Motivation and results}

There has been considerable recent interest in studying spaces of metrics with
various curvature restrictions, such as nonnegative sectional curvature, to be denoted \mbox{$K\ge 0$},
see~\cite{TW15} and references therein.
For a  manifold $V$ let 
$\mathcal{R}_{K\ge 0}(V)$ denote the space of complete Riemannian metrics on $V$ 
of $K\ge 0$ with the topology of smooth ($=C^\infty$) uniform convergence on compact sets,
and $\mathcal{M}_{K\ge 0}(V)$ be the corresponding moduli space,
the quotient space of $\mathcal{R}_{K\ge 0}(V)$ by the $\Diff(V)$-action via pullback.  

The soul construction~\cite{CG72} takes as the input a complete
metric of $K\ge 0$ on an open connected manifold $V$, and a basepoint of $V$,
and produces a totally convex compact boundaryless submanifold $S$ of $V$, called the {\em soul},
such that $V$ is is diffeomorphic to a tubular neighborhood of $S$.
If we fix a metric and vary the basepoint, the resulting souls are ambiently isotopic~\cite{Yi90} and isometric~\cite{Sh79}.
Consider the map {\bf soul} that sends an isometry class of a complete metric of $K\ge 0$ on $V$ 
to the isometry class of its soul:
\vspace{3pt}
\[
\text{\bf soul}\co\mathcal{M}_{K\ge 0}(V)\rightarrow \coprod_{S\in \mathcal V}\mathcal{M}_{K\ge 0}(S)
\]
where the co-domain is given the topology of disjoint union, and $\mathcal V$ is a set
of pairwise non-diffeomorphic manifolds such that $S\in \mathcal V$ if and only if
$S$ is diffeomorphic to a soul of a complete metric of $K\ge 0$ on $V$. 

A tantalizing open problem is to decide if the map {\bf soul}
is continuous; the difficulty is that the soul is constructed via asymptotic geometry
which is not captured by the compact-open topology on the space of metrics. 
The following is immediate from~\cite[Theorem 2.1]{BFK17}.

\begin{theorem} 
\label{thm: BFK cont}
If $V$ is indecomposable, then the map {\bf soul}
is continuous. 
\end{theorem}

An open manifold is {\em indecomposable\,} if it admits a complete
metric of $K\ge 0$ such that the normal sphere bundle to a soul has no section.
It follows from~\cite{Yi90} that for indecomposable $V$ the soul 
is uniquely determined by the metric (and not the basepoint).
Moreover, \cite{BFK17} implies that the souls of nearby metrics
are ambiently isotopic by a small compactly supported isotopy. 
In particular, metrics with non-diffeomorphic souls in an indecomposable
manifold lie in different connected components of $\mathcal{M}_{K\ge 0}(V)$. 

There are many examples where the diffeomorphism (or even homeomorphism) type of the soul depends on the metric, 
see~\cite{Be03, KPT05, BKS11, Ot11, BKS15, GZ18}, and if the ambient open manifold $V$ is indecomposable, this gives
examples where $\mathcal{M}_{K\ge 0}(V)$ is not connected, or even has infinitely many connected components.

If $V$ has a complete metric with $K\ge 0$ with soul of codimension one, then {\bf soul}
is a homeomorphism, see~\cite{BKS11}. Thus
if for some soul $S$ the space $\mathcal{M}_{K\ge 0}(S)$ has infinitely many connected components, then so does 
$\mathcal{M}_{K\ge 0}(V)$; for example, this applies to $V=S\times\RR$.

Examples of closed manifolds $S$ for which 
$\mathcal{M}_{K\ge 0}(S)$ has infinitely many connected 
components can be found in~\cite{KPT05, DKT18, De17, Go20a, DG19, De20}.
These metrics on $S$ have $K\ge 0$ and $\mathrm{scal}>0$, and the connected components are
distinguished by index-theoretic invariants that are constant
of paths of $\mathrm{scal}>0$.

The papers mentioned in the previous paragraph only 
prove existence of infinitely many path-components.
We take this opportunity to note that they actually get infinitely many 
connected components.

\begin{theorem}
\label{thm: path scal>0}
Let $M$ be a closed manifold. If two points 
in the same connected component of $\mathcal M_{K\ge 0}(M)$
have $\mathrm{scal}>0$, then they  
can be joined by a path of isometry classes of $\mathrm{scal}>0$. 
\end{theorem}

In this paper we show that some of these $S$ as above 
can be realized as souls of codimensions 
$2$ or $3$ in indecomposable manifolds. 
The codimension $2$ case is a fairly straightforward consequence of results in~\cite{WZ90, KS93, DKT18}.

\begin{theorem}
\label{thm: witten}
For every positive integer $n$ there are infinitely many homotopy types that 
contain a manifold $M$ such that\vspace{5pt}\newline
$\textup{(a)}$ $M$ is a simply-connected manifold 
that is the total space of a principal circle bundle over $\mathbb{S}^2\times \mathbb{C}P^{2n}$, and\vspace{5pt}\newline
$\textup{(b)}$
if $V$ is the total space of a non-trivial complex line bundle over $M$,
then $V$ has infinitely many complete metrics of $K\ge 0$ whose souls
equal the zero section, and 
whose isometry classes lie in different connected components of $\mathcal{M}_{K\ge 0}(V)$.
\end{theorem}

The codimension $3$ case requires a bit more work. Recall that 
if $M$ is the total space of a linear $\sph^3$-bundle over $\sph^4$, then $M$ admits 
a metric of $K\ge 0$~\cite{GZ00}, and moreover, if the bundle has nonzero Euler number,
then $\mathcal{M}_{K\ge 0}(M)$ has infinitely many connected components~\cite{De17, Go20a}.
We prove:

\begin{theorem}
\label{thm: cd 3}
Let $M$ be the total space of a linear $\sph^3$-bundle $\xi$ over $\sph^4$ with
Pontryagin number $p_1(\xi)$ and nonzero Euler number $e(\xi)$.
If $\frac{p_1(\xi)}{2e(\xi)}$ is not an odd integer, 
then $M$ is diffeomorphic to a codimension three submanifold $S$ of an indecomposable manifold $V$
that admits infinitely many complete metrics of $K\ge 0$ with soul $S$ whose
isometry classes lie in different connected components of $\mathcal{M}_{K\ge 0}(V)$.
\end{theorem}

Milnor famously showed that some $\sph^3$-bundles over $\sph^4$ are exotic spheres~\cite{Mi56}.
In fact, $M$ is a homotopy sphere if and only if $e(\xi)=\pm 1$. Unfortunately, if 
$e(\xi)=\pm 1$, then $\frac{p_1(\xi)}{2}$ is an odd integer, so no $M$ in 
Theorem~\ref{thm: cd 3} is a homotopy sphere. On the other hand,
for every integer $n$ with $n\ge 2$ there is $M$ as in the conclusion of Theorem~\ref{thm: cd 3}
with  $H^4(M)\cong\ZZ_n$, see Section~\ref{sec: codim 3}. 

To prove Theorem~\ref{thm: cd 3} we use results of Grove-Ziller~\cite{GZ00}
and some topological considerations to find an indecomposable $V$ with a codimension three soul,
and then we observe that the metric on the soul can be moved by Cheeger deformation to 
metrics in~\cite{De17, Go20a} that represent infinitely many connected components.

Let us conclude by mentioning that other results on connected components of
moduli spaces corresponding to various nonnegative or positive curvature conditions 
can be found in~\cite{KS93, BG95, Wr11, TW17, Go20b}.

\subsection*{Structure of the paper} 
Theorems~\ref{thm: BFK cont} and~\ref{thm: path scal>0} are proved in Section~\ref{sec: connect comp}.
Theorem~\ref{thm: witten} is established in Section~\ref{sec: codim2}.
Theorem~\ref{thm: cd 3} is proved in Section~\ref{sec: codim 3}, and the needed background 
is reviewed in Sections~\ref{sec: cheeger}, \ref{sec: 3-sphere}, \ref{sec: bundles}.

\subsection*{Acknowledgements} The authors are grateful to Luis Guijarro for hospitality during
Belegradek's visit to Madrid where this project was started.

\section{Continuity of souls, connectedness and path-connectedness}
\label{sec: connect comp}

\begin{proof}[Proof of Theorem~\ref{thm: BFK cont}]
Theorem 2.1 in~\cite{BFK17} says that the map that sends  
a complete metric of $K\ge 0$ on $V$ to its soul, considered as a point in the space of smooth compact submanifolds of $V$ with smooth topology, is continuous. 
Two nearby submanifolds are ambiently isotopic by a small isotopy with compact support.
Hence, the isometry classes of the induced metrics on these submanifolds are close in the moduli space.
Thus we get a continuous map
\vspace{3pt}
\[
\mathcal{R}_{K\ge 0}(V)\rightarrow \coprod_{S\in \mathcal V}\mathcal{M}_{K\ge 0}(S)
\]
that takes a metric to the isometry class of its soul, where the co-domain is given
the disjoint union topology, i.e., 
the set in the co-domain is open if and only if its intersection with each $\mathcal{M}_{K\ge 0}(S)$ is open. 
Finally, by the definition of quotient topology the above continuous map 
descends to a continuous map defined on $\mathcal{M}_{K\ge 0}(V)$. 
\end{proof}

Let $X$ denote the the space of isometry classes of Riemannian metrics on a closed manifold $M$
with smooth ($=C^\infty$) topology, and let $X_{\mathrm{scal}\ge 0}$, $X_{\mathrm{scal}>0}$
be the subspaces of $X$ of isometry classes of metrics of nonnegative and positive scalar curvature, respectively.

\begin{lemma}
$X$ is metrizable.
\end{lemma}
\begin{proof}
This is well-known, but
we cannot find a proof in the literature, and hence present it here for completeness. 
The smooth topology on the space of all Riemannian metrics
on $M$ is induced by a metric 
whose isometry group contains $\Diff(M)$~\cite[Proposition 148]{Ebin-thesis}, and 
every $\Diff(M)$-orbit is closed~\cite[Proposition 142]{Ebin-thesis}.
The corresponding pseudometric on the set of orbits
induces the quotient topology, and the pseudodistance is simply the infimum
of distances between the orbits~\cite[Theorem 4]{Hi68}. Since the orbits are closed,
the quotient space is $T_1$, so that the pseudometric is actually a metric.
\end{proof}

Also $X$ is locally path-connected (because this property is inherited by quotients, and
$X$ is the quotient of the space of metrics, which is an open subset in the Fr\'echet space
of $2$-tensors on $M$). 
In fact, every point of $X$ has a contractible neighborhood 
(as follows from the smooth version of Corollary 7.3 in~\cite{Ebin-symp} which can be deduced
from the discussion after the corollary) but we do not need it here. 

\begin{theorem} 
\label{thm: path scal >=0}
If $C$ is a connected subset of $X_{\mathrm{scal}\ge 0}$ that contains no Ricci-flat metrics,
then any two points $y, z\in C$ can be joined by a path in $\{y, z\}\cup X_{\mathrm{scal}>0}$. 
\end{theorem}
\begin{proof}
By continuous dependence of Ricci flow on initial metric, see e.g.~\cite[Theorem A]{BGI}, for 
every point $x\in X$ there is a neighborhood $U_x$ and a positive constant $\tau_x$
such that the Ricci flow that starts at any point of $U_x$
exists in $[0,\tau_x]$. 

Being a metrizable space $X$ is paracompact Hausdorff, and hence 
has a locally finite open cover $\{R_{x_i}\}_{i\in I}$
such that $R_{x_i}\subset U_{x_i}$ for all $i$, 
and there is a continuous function $\tau\co X\to (0,\infty)$ with $\tau(x)\le \tau_{x_i}$ 
for all $x\in R_{x_i}$, see~\cite[Theorem 41.8]{Mun-book}. 
Since $C$ contains no Ricci-flat metrics, for every $x\in C$ the Ricci flow of $x$ 
has $\mathrm{scal}>0$ for all times in $(0,\tau(x)]$, see~\cite[Proposition 2.18]{Br10}.
By continuous dependence of the Ricci flow on initial metric
the map $T\co X\to X$ that sends $x$ to the Ricci flow of $x$ at time $\tau (x)$
is continuous.
Hence, if $C$ is a connected subset of $X_{\mathrm{scal}\ge 0}$
that contains $y, z$, then $T(C)$ is a connected subset 
of $X_{\mathrm{scal}>0}$.

Since $X_{\mathrm{scal}>0}$ is an open subset in the locally path-connected space $X$,
every connected component of $X_{\mathrm{scal}>0}$ is path-connected. Hence the connected
component of $X_{\mathrm{scal}>0}$ that contains $T(C)$ also contains
a path from $T(y)$ to $T(z)$. Concatenating the path with Ricci flows from
$y, z$ to $T(y), T(z)$, respectively, we get a path from $y$ to $z$ with desired properties. 
\end{proof}

\begin{proof}[Proof of Theorem~\ref{thm: path scal>0}]
No flat manifold admits a metric of $\mathrm{scal}>0$~\cite[Corollary A]{GL83}.
Hence $M$ admits no flat metric. 
Since Ricci-flat metrics of $K\ge 0$ are flat,
$\mathcal M_{K\ge 0}(M)$ contains no Ricci-flat metrics.
Applying Theorem~\ref{thm: path scal >=0} to the 
connected component of $\mathcal M_{K\ge 0}(M)$ that contains $y, z$
finishes the proof. 
\end{proof}

\section{Codimension two}
\label{sec: codim2}

\begin{proof}[Proof of Theorem~\ref{thm: witten}]
If  $\mathbb{S}^{2t+1}\to \mathbb{C}P^t$ is the circle bundle
obtained by restricting the diagonal circle action on $\CC^{t+1}$, where $t$ is a positive integer, then
its Euler class generates $H^2(\mathbb{C}P^t)$ as follows from the Gysin sequence and
$2$-connectedness of $\mathbb{S}^{2t+1}$. 
Consider the product of two such circle bundles with $t=1$ and $t=2n$.
Then the argument~\cite[p.227]{WZ90} 
implies that any $M$ as in (a) is the quotient of the Riemannian product of two unit spheres 
$\sph^3\times \sph^{4n+1}$ by the free isometric circle action
$e^{i\phi}(x,y)=(e^{il\phi} x,e^{-ik\phi}y)$ for some coprime integers $k$, $l$. 
This gives a Riemannian submersion metric on $M$ with $K\ge 0$ and $\mathrm{Ric}>0$. 

Sometimes it happens that the quotients corresponding to different pairs $(k, l)$ are diffeomorphic.
In fact, $H^4(M)$ is a cyclic group of order $l^2$, so up to sign $l$ is determined by the homotopy type of $M$, but
for a given $l$ the quotients fall into finitely many diffeomorphism types~\cite[Proposition 2.2]{DKT18}.
Their diffeomorphism classification was studied in~\cite{WZ90, KS93} and finally in~\cite{DKT18} where it was shown
that for each $n$ there are infinitely many homotopy types 
that contain $M$ as in (a) and such that the Riemannian submersion metrics as above
represent infinitely many connected components of $\mathcal{M}_{K\ge 0}(M)$.

Similarly, since $\sph^3\times \sph^{4n+1}$ is $2$-connected, 
any complex line bundle over $M$ is the quotient of $\sph^3\times \sph^{4n+1}\times\CC$
by the circle action $e^{i\phi}(x,y, z)=(e^{il\phi} x,e^{-ik\phi}y, e^{im}z)$, cf.~\cite[Lemma 12.3]{BKS15}.
In particular, $V$ carries a complete Riemannian submersion metric of $K\ge 0$ with soul 
equal to the zero section, which is the quotient of $\sph^3\times \sph^{4n+1}\times\{0\}$ by the above circle action, 
and hence is diffeomorphic to $M$.
   
If we fix $l$ and the Euler class of the line bundle in $H^2(M)\cong\ZZ$, 
there are only finitely many possibilities for the diffeomorphism type of the pair $(V, \text{\,soul})$
for the above metrics.
By varying $k$ appropriately, then we get a sequence of complete metrics of $K\ge 0$ on each $V$ as above
such that the metrics on the soul represent infinitely many connected components of $\mathcal{M}_{K\ge 0}(M)$.

If the line bundle is non-trivial, then $V$ is indecomposable, and the map {\bf soul}
is continuous by Theorem~\ref{thm: BFK cont}. 
Thus $\mathcal{M}_{K\ge 0}(V)$ has infinitely many connected components.
\end{proof}

\section{Equivariant Cheeger deformation}
\label{sec: cheeger}

The purpose of this section is to review the Cheeger deformation, and note that it passes to quotients
by free isometric actions.

Let $G$ be a compact Lie group with a bi-invariant metric $Q$ that acts isometrically on 
a Riemannian manifold $(M,q_0)$. Consider the diagonal $G$-action on $M\times G$
given by $a\cdot (p,g)=(ap, ag)$, $p\in M$, $a, g\in G$.
Its orbit space is commonly denoted by $M\times_G G$.
The map $\pi\co M\times G\to M$ given by $\pi(p,g)=g^{-1}p$ descends 
to a diffeomorphism $\phi\co M\times_G G\to M$.

For any positive scalar $t$
the $G$-action is isometric in the product metric $q_0+\frac{Q}{t}$, which 
induces a metric $q_t$ on $M$ that makes $\pi$ into a Riemannian submersion. 
Similarly, $\phi$ becomes an isometry between $q_t$, $t>0$, and the Riemannian submersion 
metric on $M\times_G G$ induced by $q_0+\frac{Q}{t}$.

The map $t\to g_t$ is continuous for $t\ge 0$; this is  
the {\em Cheeger deformation} of $q_0$, see e.g.~\cite[p.140]{AB15} or~\cite{Zi09}. 
The key property is that if $q_0$ has $K\ge 0$, then so does $q_t$ for all $t$.

Fix a closed subgroup $H$ of $G$ such that the $H$-action on $M$ is free.
For $t\ge 0$ let $\bm{q_t}$ be the metric
on $M/H$ that makes the $H$-orbit map into a Riemannian submersion
$\chi\co(M, q_t)\to (M/H, \bm{q_t})$.
The map $t\to \bm{q_t}$ is continuous for $t\ge 0$.

The $H$-action on $M\times G$ given by $h\cdot (p,g)=(p, gh^{-1})$
commutes with the diagonal $G$-action, 
and hence descends to a free $H$-action on $M\times_G\hspace{1pt} G$. 
For this action the maps $\pi$ and $\phi$ are $H$-equivariant, and
descend to a Riemannian submersion 
$M\times (H\backslash G)\to M/H$ and an isometry $M\times_G (H\backslash G)\to M/H$, 
respectively, where $t>0$ and 
$H\backslash G$ is given the Riemannian submersion metric induced by $\frac{Q}{t}$. 

Thus
in the following diagram all maps are Riemannian submersions for $t>0$
\begin{equation*}
\xymatrix{
& M\times G\ar[dl]\ar[r]\ar[d]^\pi    & M\times H\backslash G\ar[d]\ar[rd] &\\
M\times_G G\ar[r]^{\quad\phi} & M \ar[r]^{\chi} & M/H  & M\times_G H\backslash G \ar[l]
}
\end{equation*}
and $\chi$ is also a Riemannian submersion for $t=0$. In this diagram $M$ and $M/H$
are the only spaces where the metric corresponding to $t=0$ is defined.

\section{Some algebra and geometry of the $3$-sphere}
\label{sec: 3-sphere}

In this section we specialize the discussion of Section~\ref{sec: cheeger} to
the case when $G=\sph^3\times\sph^3$, where $\sph^3$ is thought of as unit quaternions, 
and $H$ is the diagonal subgroup of $G$, i.e., $H=\{(g,g)\,:\, g\in G\}$.

Consider the diffeomorphism $\psi\co\sph^3\to H\backslash G$ given by $\psi(c)=(c,1)H$;
thus $\psi^{-1}$ sends the coset $(a,b)H$ to $ab^{-1}$.
With this identification the (left) $G$-action on $H\backslash G$
becomes $(a,b)\cdot c=acb^{-1}$, where $a, b, c\in\sph^3$; indeed
\[(a,b)(c,1)H=(ac,b)H=(acb^{-1},1)H.\]
Since $(-1,-1)$ acts trivially, the $G$-action on $H\backslash G$ descends to an $SO(4)$ action
with isotropy subgroups isomorphic to $SO(3)$. 

It follows that any $G$-invariant Riemannian
metric on $H\backslash G$ is isometric to a round $3$-sphere (i.e., a metric sphere in $\RR^4$). 
Indeed, $SO(3)$ acts transitively
on every tangent $2$-sphere, so $G$ acts transitively on the unit tangent bundle,
and hence the metric has constant Ricci curvature, which on the $3$-sphere
makes the metric round.

The discussion in Section~\ref{sec: cheeger} immediately gives the following.

\begin{prop}
\label{prop: cheeger}
Let $H$ be the diagonal subgroup of $G=\sph^3\times\sph^3$.
Given an isometric $G$-action on a Riemannian manifold $(M, q_0)$ of $K\ge 0$
that restricts to a free $H$-action 
there is path of Riemannian metrics $(M, q_t)$ of $K\ge 0$, defined for $t\ge 0$, such that
\vspace{3pt}\newline
$\bullet$
for every $t\ge 0$ the $G$-action is $q_t$-isometric, and the Riemannian
submersion metric $(M/H, \bm{q_t})$ induced by $q_t$ has $K\ge 0$,
and $t\to \bm{q_t}$ is a continuous path of metrics on $M/H$,
\vspace{3pt}\newline
$\bullet$
if $t>0$ and $H\backslash G$ is given the Riemannian
submersion metric induced by a bi-invariant metric on $G$, 
then $H\backslash G$ is isometric to a round sphere, and 
$(M/H,\bm{q_t})$ is isometric to the Riemannian
submersion metric on $(M, q_t)\times_G H\backslash G$.
\end{prop}

\section{Bundle theoretic facts}
\label{sec: bundles}

This section reviews several well-known bundle theoretic facts.

\begin{lemma} 
\label{lem: class spaces}
Let $C\le G$ be an order two normal subgroup of a topological group $G$. 
If $P\to X$ is a non-trivial principal $G$-bundle over a finite cell complex  with $H^1(X;\ZZ_2)=0$, 
then the associated principal $G/C$-bundle $P/C\to X$ is non-trivial.
\end{lemma}
\begin{proof} 
The surjection $G\to G/C=H$ induces a fibration of classifying spaces $BC\to BG\to BH$ 
where $BC$ is a homotopy fiber of $BG\to BH$, see~\cite{MO-fibration-classifying-spaces}. 
As explained in~\cite[p.139]{MT68}, for any finite complex $X$ we get an exact sequence of pointed sets 
\[
[X, BC]\to [X,BG]\to [X,BH]
\]

with constant maps as basepoints. Since $[X, BC]=H^1(X;\ZZ_2)=0$,
the rightmost arrow is injective.
\end{proof}

A {\em $k$-plane bundle\,} is a vector bundle with fiber $\RR^k$.

\begin{lemma}
\label{lem: 3-plane bundles}
Let $X$ be a paracompact space with $H^1(X;\ZZ_2)=0=H^2(X)$. If a 
$3$-plane bundle over $X$ has a nowhere zero section, then it is trivial. 
\end{lemma}
\begin{proof} A nowhere zero section gives rise to a splitting of the bundle into the Whitney sum
of a line and a $2$-plane subbundles, which are orientable since $H^1(X;\ZZ_2)=0$, and in fact, trivial
because a line bundle is determined by its first Stiefel-Whitney class in $H^1(X;\ZZ_2)$, and
an orientable $2$-plane bundle is determined by its Euler class in $H^2(X)$.   
\end{proof}

\begin{lemma}
\label{lem: euler pontr}
If $X$ is a finite cell complex with $H^1(X;\ZZ_2)=0=H^4(X;\QQ)$, then 
the number of isomorphism classes of $3$-plane bundles over $X$ is finite.
\end{lemma}
\begin{proof}
Since $H^1(X;\ZZ_2)=0$, any vector bundle over $X$ is orientable.
There are only finitely many isomorphism classes of orientable $3$-plane bundles 
with a given first rational Pontryagin class~\cite[Theorem A.0.1]{Be01}, which lies in
$H^4(X;\QQ)=0$.  
\end{proof}

\section{Codimension three}
\label{sec: codim 3}

This section ends with a proof of Theorem~\ref{thm: cd 3}. First, we recall some results and notations
from~\cite{GZ00}. 

Following~\cite[p.349]{GZ00} let $P_{k,l}$ denote the principal $\sph^3\times \sph^3$-bundle over $\sph^4$
classified by the map $q\to (q^k, q^{-l})$ in $\pi_3(\sph^3\times \sph^3)\cong\ZZ\times\ZZ$, where $q\in \sph^3$.

Let $M_{k,l}$ be the the associated bundle
$P_{k,l}\times_{\sph^3\times \sph^3} \sph^3$ where the action on $\sph^3$ is 
as in Section~\ref{sec: 3-sphere}, see~\cite[p.352]{GZ00}.
Equivalently~\cite[Proposition 8.27]{Po95}, the action is given by the universal covering
$\sph^3\times \sph^3\to SO(4)$ where the $SO(4)$-action on $\sph^3$ is standard,
Hence, $M_{k,l}$ is a linear $\sph^3$-bundle over $\sph^4$.

The Euler number and the Pontryagin number of the $\sph^3$-bundle $M_{k,l}\to \sph^4$ are
$\pm(k+l)$ and $\pm 2(k-l)$, see~\cite[p.159, 169]{Kr10}. 
The Gysin sequence shows that $H^4(M_{k,l})\cong\ZZ_{k+l}$ if $k+l\neq 0$, and 
then $H^4(M_{k,l})\cong\ZZ$ if $k+l=0$.


\begin{remark}
\rm Somewhat confusingly, the notation $M_{m,n}$ is also used in the literature to denote the total space of 
another $\sph^3$-bundle over $\sph^4$ based on a different choice of generators
in $\pi_3(\sph^3\times \sph^3)$. This usage goes back to James and Whitehead, and
more to the point, appears in works quoted below. Thus 
$M_{k,l}$ of~\cite{GZ00} equals $M_{m,n}$ of~\cite{CE03, Go20a} when $m=-l$, $n=k+l$. In what follows all results
are rephrased in notations of~\cite{GZ00}.  
\end{remark}
 
According to Section~\ref{sec: cheeger} $M_{k,l}$ can be described as $P_{k,l}/H$ 
where $H$ is the diagonal subgroup in $\sph^3\times \sph^3$, cf. also Key Observation in~\cite{GKS20}.
Thus $M_{k,l}$ is the base of a principal $\sph^3$-bundle with total space $P_{k,l}$.
Our strategy hinges on the following:

\begin{problem} Find all $k, l$ such that
the principal $H$-bundle $P_{k,l}\to P_{k,l}/H=M_{k,l}$ is non-trivial.
\end{problem}

Some partial solutions are presented below. An especially interesting case (which we could not resolve in this paper)
is when $|k+l|=1$, or equivalently, $M_{k,l}$ is a homotopy sphere. 

\begin{lemma}
If $kl=0$, the principal $H$-bundle $P_{k,l}\to M_{k,l}$ is trivial.
\end{lemma} 
\begin{proof}
The principal $\sph^3\times\sph^3$-bundle
$P_{k,0}$ is isomorphic to $P\times \sph^3$ for some principal $\sph^3$-bundle 
$P$ over $\sph^4$. The inclusion $i\co P\to P\times \sph^3$ given by $i(p)=(p,1)$
is transverse to the $H$-orbits, hence it descends to an immersion 
$\bm{i}\co P\to (P\times \sph^3)/H$, which is a diffeomorphism 
because both domain and co-domain are closed manifolds of the same dimension. 
Then $i\circ\bm{i}^{-1}$ is a section of 
$P\times \sph^3\to (P\times \sph^3)/H$, and 
any principal bundle with a section is trivial.
\end{proof}

Lemma~\ref{lem:Milnor} below sheds some light on why the 
assumption ``$\frac{p_1(\xi)}{2e(\xi)}$ is not odd'' is relevant. Let us
first restate the assumption:

\begin{lemma}
\label{lem: odd}
$\frac{k-l}{k+l}$ is an odd integer if and only if $\frac{k}{k+l}\in\ZZ$.
\end{lemma}
\begin{proof}
$\frac{k-l}{k+l}$ is odd if and only if $\frac{k-l}{k+l}+1=\frac{2k}{k+l}$ is even
if and only if $\frac{k}{k+l}\in\ZZ$
\end{proof}

\begin{lemma}
\label{lem:Milnor}
If $kl\neq 0$ and the principal $H$-bundle $P_{k,l}\to P_{k,l}/H=M_{k,l}$ is trivial, then $k+l\neq 0$ 
and $\frac{k}{k+l}\in\ZZ$.
\end{lemma}
\begin{proof}
If the bundle is trivial, 
$P_{k,l}$ is diffeomorphic to $\sph^3\times M_{k,l}$. By the K\"unneth formula 
$H^4(P_{k,l})\cong H^4(M_{k,l})$ which is $\ZZ_{k+l}$ if $k+l\neq 0$, and 
$\ZZ$ if $k+l=0$. 
As was mentioned on~\cite[p.349]{GZ00}, the quotient of the principal $\sph^3\times \sph^3$-bundle $P_{k,l}$
by the subgroup $1\times \sph^3$ can be identified with $P_k$, 
the principal $\sph^3$-bundle over $\sph^4$ with Euler number $k$. 
Since $k\neq 0$, we get
$H^4(P_k)\cong\ZZ_{k}$~\cite[p.346]{GZ00}.
The Gysin sequence for the $\sph^3$-bundle $P_{k,l}\to P_k$ reads
\begin{equation*}
\xymatrix{
\ZZ_k\cong H^4(P_k)\ar[r] & H^4(P_{k,l})\ar[r] & H^1(P_k)=0,
}
\end{equation*}
which shows that $k+l$ is a nonzero integer that divides $k$.
\end{proof}

\begin{remark}\rm
The asymmetry in the conclusion of Lemma~\ref{lem:Milnor} is an illusion:
$\frac{k}{k+l}\in\ZZ$ if and only if $\frac{l}{k+l}\in\ZZ$ because
$\frac{k}{k+l}+\frac{l}{k+l}=1$.
\end{remark}

\begin{proof}[Proof of Theorem~\ref{thm: cd 3}]
By Proposition 3.11 of~\cite{GZ00} each $P_{k,l}$ admits a cohomogeneity one action by 
$\mathbb{S}^3\times \mathbb{S}^3\times \mathbb{S}^3$ with codimension  two singular orbits, and such that
the action of the subgroup $G:=\mathbb{S}^3\times \mathbb{S}^3\times \{1\}$ coincides with the principal bundle action.
Hence by \cite[Theorem E]{GZ00} the space $P_{k,l}$ carries a $G$-invariant metric $\gamma_{k,l}$ of $K\ge 0$.

Let $\sph^3(r)$ be the round $3$-sphere of radius $r$ on which $\sph^3\times\sph^3$ acts  
as in Section~\ref{sec: 3-sphere}. Let $h_{k,l, r}$ be the Riemannian submersion metrics on
$M_{k,l}=P_{k,l}\times_{\sph^3\times\sph^3} \sph^3$ induced by the product of $\gamma_{k,l}$ and $\sph^3(r)$.
Then $h_{k,l, r}$ has $K\ge 0$ and $\mathrm{scal}>0$ by~\cite[Theorem 2.1]{Go20a}.

An essential point is that there are infinitely many ways to represent $M$ as $M_{k,l}$. 
Indeed, by assumption $M=M_{k,l}$ for some $k,l\in\ZZ$ with $k+l\neq 0$ and such that
$\frac{k-l}{k+l}$ is not an odd integer. The latter is equivalent to
$\frac{k}{k+l}\notin\ZZ$ by Lemma~\ref{lem: odd}. For $i\in\ZZ$ 
let \[
l_i=l-56(k+l)i\quad\text{and}\quad k_i=k+l-l_i=-l+(k+l)(56i+1).
\] 
Then $M_{k_i, l_i}$ are 
orientation-preserving diffeomorphic to $M$~\cite[Corollary 1.6]{CE03}. 
By~\cite[Section 3.1]{Go20a} there is $r$ and infinitely many values of $i$
for which the metrics $h_{k_i,l_i, r}$ lie in different connected components
of $\mathcal{M}_{K\ge 0}(M)$. 

Let $g_{k,l}$ be the Riemannian submersion metric of 
$K\ge 0$ induced on $M_{k,l}=H\backslash P_{k,l}$ by $\gamma_{k,l}$.
Proposition~\ref{prop: cheeger} implies that $g_{k,l}$ and $h_{k,l, r}$ lie
in the same path-component of $\mathcal{M}_{K\ge 0}(M_{k,l})$.
Thus for $k_i,l_i$ as in the previous paragraph $g_{k_i,l_i}$ 
lie in different connected components
of $\mathcal{M}_{K\ge 0}(M)$. 

Consider the associated vector bundle $P_{k,l}\times_{H}\RR^3$ over $M_{k,l}$
where $H=\sph^3$ acts on $\RR^3$ via the universal covering $\sph^3\to SO(3)$.
We give $P_{k,l}\times_{H}\RR^3$ the Riemannian submersion metric induced by
the product of $\gamma_{k,l}$ and the standard Euclidean metric.
This is a complete metric of $K\ge 0$ with soul $P_{k,l}\times_{H}\{0\}$ which
is isometric to $(M_{k,l}, g_{k,l})$.

Since $\frac{k_i}{k_i+l_i}\notin\ZZ$, the principal $\sph^3$-bundle $P_{k_i,l_i}\to M_{k_i,l_i}$ is non-trivial 
by Lemma~\ref{lem:Milnor}. 
Consider the associated $3$-plane bundle $P_{k_i,l_i}\times_{\sph^3}\RR^3$ over $M_{k_i,l_i}$
where $\sph^3$ acts on $\RR^3$ via the universal covering $\sph^3\to SO(3)$.
Any such vector bundle is non-trivial by Lemma~\ref{lem: class spaces}, and hence 
by Lemma~\ref{lem: 3-plane bundles}
its total space is indecomposable.
Pull back the vector bundles via diffeomorphisms $M\to M_{k_i,l_i}$.
The pullback bundles fall into finitely many isomorphism classes by Lemma~\ref{lem: euler pontr},
so after passing to a subsequence we can assume that the bundles are isomorphic,
and hence share the same ten-dimensional total space, which we denote $V$.
   
In summary, $V$ is an indecomposable manifold with infinitely many complete metrics of $K\ge 0$
whose souls are all equal to the zero section, and diffeomorphic to $M$, and such that
the induced metrics on the souls lie in different connected components of $\mathcal{M}_{K\ge 0}(M)$. 
Theorem~\ref{thm: BFK cont} finishes the proof.
\end{proof}

\small

\providecommand{\bysame}{\leavevmode\hbox to3em{\hrulefill}\thinspace}
\providecommand{\MR}{\relax\ifhmode\unskip\space\fi MR }
\providecommand{\MRhref}[2]{%
  \href{http://www.ams.org/mathscinet-getitem?mr=#1}{#2}
}
\providecommand{\href}[2]{#2}

\end{document}